\documentclass[leqno,12pt]{amsart}
\usepackage{amsfonts,amsmath,amssymb,enumerate,graphicx,psfrag,xcolor,tikz,float,hyperref}

\usetikzlibrary{arrows.meta,calc}
\definecolor{mygray}{RGB}{211,215,207}
\definecolor{lightgray}{RGB}{215,215,215}

\newtheorem{theorem}{Theorem}[section]
\newtheorem{corollary}[theorem]{Corollary}
\newtheorem{lemma}[theorem]{Lemma}
\newtheorem{proposition}[theorem]{Proposition}

\usepackage{hyperref}
\hypersetup{
    colorlinks=true,
    linkcolor= blue,
    citecolor =cyan,
    urlcolor = teal,
}

\theoremstyle{definition}
\newtheorem{definition}[theorem]{Definition}
\newtheorem{remark}[theorem]{Remark}

\numberwithin{equation}{section}

\usepackage{amsfonts,amsmath,amssymb,hyperref,xcolor,float,tikz,hyperref,enumerate}
\usetikzlibrary{arrows.meta,calc}
\definecolor{mygray}{RGB}{211,215,207}
\definecolor{lightgray}{RGB}{215,215,215}

\DeclareMathOperator{\supp}{supp}

\begin{document}

\title[Atomic decomposition in the Dunkl setting ]{Remark on atomic decompositions \\
for Hardy space $H^1$\\
 in the rational Dunkl setting}

\author[ J. Dziuba\'nski and A. Hejna]{Jacek Dziuba\'nski and Agnieszka Hejna}

\date{}

\address{J. Dziuba\'nski and A. Hejna, Uniwersytet Wroc\l awski,
Instytut Matematyczny,
Pl. Grunwaldzki 2/4,
50-384 Wroc\l aw,
Poland}
\email{jdziuban@math.uni.wroc.pl}
\email{hejna@math.uni.wroc.pl}

\thanks{The research supported by the National Science Centre, Poland (Narodowe Centrum Nauki), Grant 2017/25/B/ST1/00599.
}

\maketitle


\renewcommand{\thefootnote}{}

\footnote{2010 \emph{Mathematics Subject Classification}: Primary\,: 42B30.
Secondary\,: 42B25, 35K08, 42B25,  42B35, 42B37.}

\footnote{\emph{Key words and phrases}: Dunkl operator,  Hardy spaces, maximal operator, atomic decomposition.}

\renewcommand{\thefootnote}{\arabic{footnote}}
\setcounter{footnote}{0}


\begin{abstract}
Let $\Delta$  be the Dunkl Laplacian on $\mathbb R^N$ associated with a  normalized root system $R$ and a multiplicity function $k(\alpha)\geq 0$. We say that a function $f$ belongs to the Hardy space $H^1_{\Delta}$  if the nontangential maximal function defined by $\mathcal M_H f(\mathbf x)=\sup_{\| \mathbf x-\mathbf y\|<t} |\exp(t^2\Delta )f(\mathbf x)|$ belongs to $L^1(w(\mathbf x)\, d\mathbf x)$, where $w(\mathbf x)=\prod_{\alpha\in R} |\langle \alpha,\mathbf x\rangle|^{k(\alpha)}$. We prove that $H^1_\Delta$ admits atomic decompositions into atoms in the sense of Coifman--Weiss  on the space of homogeneous type $\mathbb{R}^N$ equipped with the Euclidean distance $\|\mathbf{x}-\mathbf{y}\|$ and the measure $w(\mathbf{x})d\mathbf{x}$. To this end we improve estimates for the heat kernel of $e^{t\Delta}$.
\end{abstract}

\section{Introduction}
Let $\Delta$  be the Dunkl Laplacian on $\mathbb R^N$ associated with a reduced normalized root system $R$ and a multiplicity function $k(\alpha)\geq 0$.
We denote by $dw(\mathbf x)=w(\mathbf x)\, d\mathbf x$, where
 \begin{equation}\label{eq:w}
 w(\mathbf x)=\prod_{\alpha\in R} |\langle \alpha,\mathbf x\rangle|^{k(\alpha)},
 \end{equation}
 is the associated measure on $\mathbb R^N$. Let
 $H_t=e^{t\Delta}$ be the Dunkl heat semigroup. The operators $H_t$ form strongly continuous semigroups of linear contractions on $L^p(dw)$ for $1\leq p<\infty$, which are self-adjoint operators on $L^2(dw)$. Moreover, the maximal operator $\sup_{t>0} |\exp(t\Delta)f(\mathbf x)|$ is bounded from $L^1(dw)$ into $L^{1,\infty}(dw)$ (see \cite[Theorems 6.1 and 6.2]{TX}).

 We say that an $L^1(dw)$-function $f$ belongs to the real Hardy space $H^1_{\Delta}$  if the nontangential maximal function
\begin{align*}
 \mathcal M_H f(\mathbf x)=\sup_{\| \mathbf x-\mathbf y\|<t} |\exp(t^2\Delta )f(\mathbf y)|
\end{align*}
belongs to $L^1(dw)$. The space $H^1_{\Delta}$ is a Banach space with the norm
\begin{align*}
\| f\|_{H^1_{{\rm max}, H}}=\| \mathcal M_H f\|_{L^1(dw)}.
\end{align*}
  In \cite{conjugate} characterizations  of $H^1_{\Delta}$  by relevant Riesz transforms, Littlewood-Paley square functions, and atomic decompositions were proved. Let us recall the notions of atoms considered in \cite{conjugate}. For  a positive integer $M$, let $\mathcal D(\Delta^M)$ denote the domain of $\Delta^M$ as an (unbounded) operator on $L^2(dw)$.  Let $G$ be the Weyl group of the root system $R$. Set $\mathcal O(\mathbf x)=\bigcup_{\sigma\in G} \{\sigma(\mathbf x)\}$.  Similarly, if $B$ is a Euclidean ball then $\mathcal O(B)=\bigcup_{\sigma\in G} \sigma (B)$ is the $G$-orbit of $B$.
  \begin{definition}\label{def:atom} \normalfont Let $1<q\leq \infty$ and $M$ be a positive integer.  A function $a(\mathbf x)$ is said to be a $(1,q,\Delta, M)$-atom if
 $a\in L^2(dw)$ and there is $ b\in \mathcal D(\Delta^M)$ and a Euclidean ball $ B=B(\mathbf y_0,r)$  such

$\bullet$ $a=\Delta^Mb$;

$\bullet$ $\supp\Delta^\ell b\subset \mathcal O(B)$ for $\ell=0,1,2,...,M$;

$\bullet$ $ \| (r^2\Delta)^{\ell}b\|_{L^q(dw)} \leq r^{2M}w(B)^{\frac{1}{q}-1}$, $\ell=0,1,...,M$.
\end{definition}

\begin{definition}\label{def:atomic_space} \normalfont A function $f$ belongs to $H^1_{(1,q,\Delta, M)}$ if there are $(1,q,\Delta, M)$-atoms  $a_j$ and $\lambda_j\in \mathbb C$ such that $f=\sum_{j=1}^{\infty} \lambda_ja_j$ and $\sum_{j=1}^{\infty}|\lambda_j|<\infty$. Then we set
$ \| f\|_{H^1_{(1,q,\Delta, M)}} =\inf \Big\{ \sum_{j=1}^{\infty} |\lambda_j|\Big\}, $
where the infimum is taken over all representations of $f$ as above.
\end{definition}
It was proved in \cite{conjugate} that the spaces $H^1_{\Delta}$ and $H^1_{(1,q,\Delta, M)}$ coincide and the corresponding norms are equivalent.

Let us note that the atoms considered  in~\cite{conjugate} (see Definition \ref{def:atom}) are in the spirit of~\cite{HMMLY}, which means that  they are of the form $a=\Delta^Mb$ for appropriate functions $b$. Our aim  is to prove that the Hardy space $H^1_{\Delta}$ admits other atomic decompositions, namely into atoms in the sense of Coifman--Weiss~\cite{CW} on the space of homogeneous type $(\mathbb R^N, \| \mathbf x-\mathbf y\|, dw)$.

 \begin{definition}\label{def:C-W-atoms} \normalfont
 Fix $1<q\leq \infty$.  A function $a(\mathbf x)$ is a $(1,q)$-atom if there is a Euclidean ball $B$ such that
\begin{enumerate}[(A)]
\item{$\supp a\subset B$;}\label{numitem:support}
\item{$\| a\|_{L^q(dw)}\leq w( B)^{\frac{1}{q}-1}$;}\label{numitem:size}
\item{$\int a(\mathbf x)\, dw(\mathbf x)=0$.}\label{numitem:cancel}
\end{enumerate}
 \end{definition}

\begin{definition}\label{def:space_C-W} \normalfont A function $f$ belongs to $H^1_{(1,q)}$ if there are $\lambda_j\in \mathbb C$ and $(1,q)$-atoms $a_j$ such that $f=\sum_{j=1}^{\infty} \lambda_ja_j$ and $\sum_{j=1}^{\infty}|\lambda_j|<\infty$. Then
$$ \| f\|_{H^1_{(1,q)}} =\inf \Big\{ \sum_{j=1}^{\infty} |\lambda_j|\Big\}, $$
where the infimum is taken over all representations of $f$ as above.
\end{definition}

We are now in a position to state our main result.

\begin{theorem}\label{teo:main1} For every $1<q\leq \infty$ the spaces $H^1_{\Delta}$ and $H^1_{(1,q)}$ coincide and the corresponding norms are equivalent, that is, there is a constant $C>0$ such that
\begin{align}\label{eq:main}
C^{-1} \| f\|_{H^1_{(1,q)}}\leq \| f\|_{H^1_{{\rm max}, H}}\leq C \| f\|_{H^1_{(1,q)}}.
\end{align}
\end{theorem}

\begin{remark}\normalfont Since every $(1,\infty)$-atom is a $(1,q)$-atom, it suffices to prove the first inequality in~\eqref{eq:main} for $q=\infty$.
\end{remark}

In order to prove the theorem we first derive improvements of estimates obtained in \cite{conjugate} of the heat kernel of the semigroup $e^{t\Delta}$ and, consequently, of  other  kernels associated with translations of radial functions. This is presented  in Section \ref{sec:further}. Then, in Section \ref{sec:SectionDecomp},  we  use a characterization of  $H^1_{\Delta}$ by Littlewood-Paley square functions to obtain decomposition into $(1,2)$-atoms. Finally, $(1,\infty)$-atomic decomposition is achieved by a~standard decomposition of $(1,2)$-atoms into $(1,\infty)$-atoms.

Let us remark, that if $k\equiv 0$, then the Hardy space $H^1_\Delta$ coincides with the classical real Hardy space $H^1$ on the Euclidean space $\mathbb R^N$ studied originally by Stein and Weiss~\cite{SW}, Fefferman and Stein~\cite{FS},  and Coifman \cite{Coifman}. More information concerning the classical theory of $H^p$ spaces can be found in the book~\cite{St2} and references therein.

\section{Preliminaries}
In this section we present basic facts concerning theory of the Dunkl operators.  For details we refer the reader to~\cite{Dunkl},~\cite{dJ}, ~\cite{Roesler3}, and~\cite{Roesler-Voit}.

We consider the Euclidean space $\mathbb R^N$ with the scalar product $\langle\mathbf x,\mathbf y\rangle=\sum_{j=1}^N x_jy_j
$, $\mathbf x=(x_1,...,x_N)$, $\mathbf y=(y_1,...,y_N)$. For
 a nonzero vector $\alpha\in\mathbb R^N$ the reflection $\sigma_\alpha$ with respect to the hyperplane $\alpha^\perp$ orthogonal to $\alpha$ is given by
\begin{equation}\label{eq:reflection}
\sigma_\alpha\mathbf x=\mathbf x-2\frac{\langle \mathbf x,\alpha\rangle}{\| \alpha\| ^2}\alpha.
\end{equation}
In this paper we fix a normalized root system in $\mathbb R^N$, that is, a finite set  $R\subset \mathbb R^N\setminus\{0\}$ such that   $\sigma_\alpha (R)=R$ and $\|\alpha\|=\sqrt{2}$ for every $\alpha\in R$. The finite group $G$ generated by the reflections $\sigma_\alpha$, $\alpha\in R$, is called the {\it Weyl group} ({\it reflection group}) of the root system. A~{\textit{multiplicity function}} is a $G$-invariant function $k:R\to\mathbb C$ which will be fixed and $\geq 0$  throughout this paper.

Let
\begin{equation}\label{eq:gamma}
\gamma =\frac{1}{2}\sum_{\alpha \in R} k(\alpha) \text{ and } \mathbf N=2\gamma +N.
\end{equation}
The number $\mathbf N$ is called the homogeneous dimension of the system, since
$$ w(B(t\mathbf x, tr))=t^{\mathbf N}w(B(\mathbf x,r)) \ \ \text{\rm for } \mathbf x\in\mathbb R^N, \ t,r>0,$$
where here and subsequently $B(\mathbf x, r)$ denotes the Euclidean ball centered at $\mathbf x$ with radius $r>0$. Observe that
\begin{equation}\label{eq:behavior} w(B(\mathbf x,r))\sim r^{N}\prod_{\alpha \in R} (|\langle \mathbf x,\alpha\rangle |+r)^{k(\alpha)},
\end{equation}
so $dw(\mathbf x)$ it is doubling, that is, there is a constant $C>0$ such that
\begin{equation}\label{eq:doubling} w(B(\mathbf x,2r))\leq C w(B(\mathbf x,r)) \ \ \text{\rm for } \mathbf x\in\mathbb R^N, \ r>0.
\end{equation}
Moreover, by \eqref{eq:behavior},
\begin{equation}\label{eq:growth}
C^{-1} \left(\frac{r_2}{r_1}\right)^{ N}\leq \frac{w(B(\mathbf x, r_2))}{w(B(\mathbf x, r_1))}\leq C \left(\frac{r_2}{r_1}\right)^{\mathbf N}\ \ \text{ for } 0<r_1<r_2.
\end{equation}

Given a (normalized) root system $R$ and a multiplicity function $k(\alpha)$ the {\it Dunkl operator} $T_\xi$  is the following $k$-deformation of the directional derivative $\partial_\xi$ by a  difference operator:

\begin{align*}
 T_\xi f(\mathbf x)= \partial_\xi f(\mathbf x) + \sum_{\alpha\in R} \frac{k(\alpha)}{2}\langle\alpha ,\xi\rangle\frac{f(\mathbf x)-f(\sigma_\alpha \mathbf x)}{\langle \alpha,\mathbf x\rangle}.
\end{align*}

The Dunkl operators $T_\xi$ were introduced in \cite{Dunkl}. They commute and are skew-symmetric in $L^2(\mathbb R^N,dw)$. Moreover, if $f,g\in C^1(\mathbb R^N)$ and at least one of them is $G$-invariant, then
\begin{equation}\label{eq:Leibniz}
T_\xi (fg)=(T_\xi f)\cdot g+f \cdot (T_\xi g).
\end{equation}
Let $e_j$, $j=1,2,...,N$, denote the canonical orthonormal basis in $\mathbb R^N$ and let $T_j=T_{e_j}$.

\noindent
{\bf Dunkl kernel and Dunkl transform.} For fixed $\mathbf y\in\mathbb R^N$ the {\it Dunkl kernel} $E(\mathbf x,\mathbf y)$ is a unique solution of the system
$$ T_\xi f=\langle \xi,\mathbf y\rangle f, \ \ f(0)=1.$$
In particular
\begin{equation}\label{eq:T_j_x} T_{j,\mathbf x} E(\mathbf x,\mathbf y)=y_jE(\mathbf x,\mathbf y),
\end{equation}
where here and subsequently  $T_{j,\mathbf x}$  denotes the action of $T_j$ with respect to the variable $\mathbf x$.

The function $E(\mathbf x ,\mathbf y)$ was introduced in~\cite{Dunkl1991}. It  generalizes the exponential function
 $e^{\langle \mathbf x,\mathbf y\rangle}$ and has a unique extension to a holomorphic function on
 $\mathbb C^N\times \mathbb C^N$.
 We have
\begin{enumerate}[(a)]
\item {$E(\lambda \mathbf x,\mathbf y)=E(\mathbf x,\lambda \mathbf y)=E(\lambda  \mathbf y,\mathbf x)=E(\lambda \sigma(\mathbf x),\sigma(\mathbf y))$ for all  $\mathbf x,\mathbf y\in \mathbb C^N$, $\sigma \in G$, $\lambda \in \mathbb C$;}\label{numitem:a}
\item
$E(\mathbf x,\mathbf y)>0$ for all $\mathbf x,\mathbf y\in\mathbb R^N$;
\item
$|E(-i\mathbf x,\mathbf y)|\leq 1 $ for all $\mathbf x,\mathbf y\in\mathbb R^N$;
\item
$E(0,\mathbf y)=1$ for all $\mathbf y\in \mathbb C^N$.
\end{enumerate}
Proofs of the equalities in~\eqref{numitem:a} can be found in~\cite{Dunkl1991}. The further properties listed above are direct consequences of~\cite[Proposition 5.1]{Roesler1999}. More details concerning the Dunkl kernel $E(\mathbf x,\mathbf y)$ can be found in the lecture notes~\cite{Roesler3},~\cite{Roesler-Voit} and references therein.

The {\it Dunkl transform}, which generalizes the classical Fourier transform, is defined on $L^1({dw})$ by (see \cite{dJ}, \cite{Roesler-Voit})
\begin{align*}
\mathcal{F} f(\xi)=c_k^{-1}\int_{\mathbb{R}^N}f(\mathbf{x})E(\mathbf{x},-i\xi)\, dw(\mathbf{x}),
\end{align*}
where
$$c_k =\int_{\mathbb R^N} e^{-\|\mathbf x\|^2\slash 2}\, dw(\mathbf x).$$
\smallskip
The Dunkl transform, which  is a topological automorphisms of the Schwartz space $\mathcal{S}(\mathbb{R}^N)$, has a unique  extension to  an isometric automorphism of $L^2(dw)$ and satisfies the following inversion formula, see \cite{dJ}.
For every $f\in L^1({dw})$ such that $\mathcal{F}f\in L^1({dw})$, we have
\begin{align*}
f(\mathbf{x})=(\mathcal{F})^2f(-\mathbf{x})
\text{ for all }\mathbf{x}\in\mathbb{R}^N.
\end{align*}
For $\lambda >0$, we have
$
\mathcal{F}(f_\lambda)(\xi)=\mathcal{F}f(\lambda\xi),
$
where
$f_\lambda(\mathbf{x})=\lambda^{-\mathbf{N}}f(\lambda^{-1}\mathbf{x}).$

\smallskip
\noindent
{\bf Dunkl translation{s} and Dunkl convolution.}
The {\it Dunkl translation\/} $\tau_{\mathbf{x}}f$ of a function $f\in\mathcal{S}(\mathbb{R}^N)$ by $\mathbf{x}\in\mathbb{R}^N$ is defined by
\begin{equation}\label{eq:translation}
\tau_{\mathbf{x}} f(\mathbf{y})=c_k^{-1} \int_{\mathbb{R}^N}{E}(i\xi,\mathbf{x})\,{E}(i\xi,\mathbf{y})\,\mathcal{F}f(\xi)\,{dw}(\xi).
\end{equation}
{We list below some  properties of Dunkl translations:
\begin{itemize}
\item[$\bullet$]
each translation $\tau_{\mathbf{x}}$ is a continuous linear map of $\mathcal{S}(\mathbb{R}^N)$ into itself, which extends to a contraction on $L^2({dw})$,
\item[$\bullet$]
(\textit{Identity\/})
$\tau_0=I$;
\item[$\bullet$]
(\textit{Symmetry\/})
$\tau_{\mathbf{x}}f(\mathbf{y})
=\tau_{\mathbf{y}}f(\mathbf{x})\text{ for all }\mathbf{x},\mathbf{y}\in\mathbb{R}^N, f\in\mathcal{S}(\mathbb{R}^N)$;
\item[$\bullet$]
(\textit{Scaling\/})
$\tau_{\mathbf{x}}(f_\lambda)=(\tau_{\lambda^{-1}\mathbf{x}}f)_\lambda \text{ for all }\lambda>0\,,\mathbf{x}
\in\mathbb{R}^N,\,f\in\mathcal{S}(\mathbb{R}^N)$;
\item[$\bullet$] $T_\xi(\tau_{\mathbf x} f)=\tau_{\mathbf x} (T_\xi f)$ for all $\mathbf{x},\xi \in \mathbb{R}^N$;
\item[$\bullet$]
(\textit{Skew--symmetry\/}) for all $\mathbf{x}\in\mathbb{R}^N$ and  $f,g\in\mathcal{S}(\mathbb{R}^N)$ we have
\begin{equation*}\quad
\int_{\mathbb{R}^N}\!\tau_{\mathbf{x}}f(\mathbf{y})\,g(\mathbf{y})\,dw(\mathbf{y})=\int_{\mathbb{R}^N}f(\mathbf{y})\,\tau_{-\mathbf{x}}g(\mathbf{y})\,dw(\mathbf{y}).
\end{equation*}
\end{itemize}
The latter formula allows us to define the Dunkl translations $\tau_{\mathbf{x}}f$ in the distributional sense for $f\in L^p({dw})$ with $1\leq p\leq \infty$. Further,
\begin{equation*}\quad
\int_{\mathbb{R}^N}\!\tau_{\mathbf{x}}f(\mathbf{y})\,dw(\mathbf{y})=\int_{\mathbb{R}^N}f(\mathbf{y})\,dw(\mathbf{y})
\text{ for all }\mathbf{x}\in\mathbb{R}^N,\,f\in\mathcal{S}(\mathbb{R}^N).
\end{equation*}

\smallskip

{The \textit{Dunkl convolution\/} of two reasonable functions (for instance Schwartz functions) is defined by
$$
(f*g)(\mathbf{x})=c_k\,\mathcal{F}^{-1}[(\mathcal{F}f)(\mathcal{F}g)](\mathbf{x})=\int_{\mathbb{R}^N}(\mathcal{F}f)(\xi)\,(\mathcal{F}g)(\xi)\,E(\mathbf{x},i\xi)\,dw(\xi)
$$
for all $\mathbf{x}\in\mathbb{R}^N$ or, equivalently, by}
$$
{(}f*g{)}(\mathbf{x})
=\int_{\mathbb{R}^N}f(\mathbf{y})\,\tau_{\mathbf{x}}g(-\mathbf{y})\,{dw}(\mathbf{y})
=\int f(\mathbf y) g(\mathbf x,\mathbf y)\, dw(\mathbf y),$$
where here and subsequently,
 $$g(\mathbf x,\mathbf y)=\tau_{\mathbf x}g(-\mathbf y)$$ for a reasonable function $g(\mathbf x)$ on $\mathbb R^N$.

\noindent
\textbf{Dunkl heat semigroup.} The {\it Dunkl Laplacian} associated with $G$ and $k$  is the differential-difference operator $\Delta=\sum_{j=1}^N T_{j}^2$. It  acts on $C^2(\mathbb{R}^N)$ functions by
 $$ \Delta f(\mathbf x)=\Delta_{\rm eucl} f(\mathbf x)+\sum_{\alpha\in R} k(\alpha) \delta_\alpha f(\mathbf x),$$
$$\delta_\alpha f(\mathbf x)=\frac{\partial_\alpha f(\mathbf x)}{\langle \alpha , \mathbf x\rangle} -  \frac{f(\mathbf x)-f(\sigma_\alpha \mathbf x)}{\langle \alpha, \mathbf x\rangle^2}.$$
 The operator $\Delta$ is essentially self-adjoint on $L^2(dw)$ and generates the semigroup $H_t=e^{t\Delta}$  of linear self-adjoint contractions on $L^2(dw)$. The semigroup has the form
  \begin{equation}\label{eq:heat_semigroup}
  e^{t\Delta} f(\mathbf x)=\int_{\mathbb R^N} h_t(\mathbf x,\mathbf y)f(\mathbf y)\, dw(\mathbf y),
  \end{equation}
where
 \begin{equation}\label{eq:heat_kernel} \ h_t(\mathbf x,\mathbf y)=c_k^{-1} (2t)^{-\mathbf N\slash 2}
 e^{-(\|\mathbf x\|^2+\|
 \mathbf y\|^2)\slash (4t)}E\left(\frac{\mathbf x}{\sqrt{2t}},\frac{\mathbf y}{\sqrt{2t}}\right),
 \end{equation}
see~\cite[Section 4]{Roesler1998}.
The heat kernel $h_t(\mathbf x,\mathbf y)$ is a $C^\infty$ function of all variables $\mathbf x,\mathbf y \in \mathbb{R}^N$, $t>0$ and satisfies \begin{equation}\label{eq:symmetry} 0<h_t(\mathbf x,\mathbf y)=h_t(\mathbf y,\mathbf x),
  \end{equation}
 \begin{equation}\label{eq:heat_one} \int_{\mathbb R^N} h_t(\mathbf x,\mathbf y)\, dw(\mathbf y)=1.
 \end{equation}
In particular  (see  \cite{Roesler1998}) for every $t>0$ and for every $\mathbf{x},\mathbf{y}\in\mathbb{R}^N$,
\begin{equation}\label{eq:expression1_heat_kernel}
h_t(\mathbf{x},\mathbf{y} )
=\tau_{\mathbf{x}}h_t(-\mathbf{y}),
\ \ \text{\rm
where} \ \
h_t(\mathbf{x})=\tilde{h}_t(\|\mathbf{x}\|)
=c_k^{-1}\,(2t)^{-\mathbf{N}/2}\,e^{-\frac{{\|}\mathbf{x}{\|}^2}{4t}}.
\end{equation}

\noindent
\textbf{Dunkl translations of radial functions.}
{The following specific formula was obtained by R\"osler \cite{Roesler2003}
for the Dunkl translations of (reasonable) radial functions $f({\mathbf{x}})=\tilde{f}({\|\mathbf{x}\|})$\,:}
\begin{equation}\label{eq:translation-radial}
\tau_{\mathbf{x}}f(-\mathbf{y})=\int_{\mathbb{R}^N}{(\tilde{f}\circ A)}(\mathbf{x},\mathbf{y},\eta)\,d\mu_{\mathbf{x}}(\eta)\text{ for all }\mathbf{x},\mathbf{y}\in\mathbb{R}^N.
\end{equation}
{Here}
\begin{equation*}
A(\mathbf{x},\mathbf{y},\eta)=\sqrt{{\|}\mathbf{x}{\|}^2+{\|}\mathbf{y}{\|}^2-2\langle \mathbf{y},\eta\rangle}=\sqrt{{\|}\mathbf{x}{\|}^2-{\|}\eta{\|}^2+{\|}\mathbf{y}-\eta{\|}^2}
\end{equation*}
and $\mu_{\mathbf{x}}$ is a probability measure, which is supported in $\operatorname{conv}\mathcal{O}(\mathbf{x})$.

Let
$$d(\mathbf x,\mathbf y)=\min_{\sigma \in G}\| \sigma (\mathbf x)-\mathbf y\|$$
denote the distance of the orbits $\mathcal O(\mathbf x)$ and $\mathcal O(\mathbf y)$. Since
\begin{equation}\label{eq:A(x,y,eta)>d(x,y)}
A(\mathbf x,\mathbf y,\eta)\geq d(\mathbf x,\mathbf y) \text{ for }\eta\in  \operatorname{conv}\mathcal{O}(\mathbf{x}),
\end{equation}
the formulas~\eqref{eq:expression1_heat_kernel} and \eqref{eq:translation-radial} imply (see e.g., \cite{Roesler-Voit})
\begin{gather}
h_t(\mathbf{x},\mathbf{y})\leq c_k^{-1}\,(2t)^{-\mathbf{N}/2}\,e^{-\frac{d(\mathbf{x},\mathbf{y})^2}{4t}}.
\label{eq:estimateRosler}
\end{gather}
\vspace{-2mm}

For $\mathbf x,\mathbf y\in\mathbb R^N$ and $t>0$
we set
$$V(\mathbf x,\mathbf y,t)=\max\Big(w(B(\mathbf x,t)), w(B(\mathbf y, t))\Big).$$
It was proved in~\cite[Theorem 4.3]{conjugate} that the factor $t^{\mathbf N\slash 2}$ in \eqref{eq:estimateRosler} can be replaced by $V(\mathbf x,\mathbf y,\sqrt{t})$, which gives the following estimates on the heat kernel in the spirit of analysis on spaces of homogeneous type. Let
\begin{align*}
\mathcal G_t(\mathbf x,\mathbf y)&=\frac{1}{V(\mathbf x,\mathbf y,\sqrt{t})}\sum_{\sigma\in G}\exp\Big(\frac{-\| \mathbf x-\sigma( \mathbf y)\|^2}{t}\Big)\\&\sim\frac{1}{V(\mathbf x,\mathbf y,\sqrt{t})}\exp\Big(-\frac {d(\mathbf x,\mathbf y)^2}{t}\Big).
\end{align*}
There are constants $C,c>0$ such that
\begin{equation}\label{heat_radial}
h_t(\mathbf x,\mathbf y)\leq C \mathcal G_{t\slash c}(\mathbf x,\mathbf y).
\end{equation}
Note that $V(\mathbf{x},\mathbf{y},t)$ and $\mathcal{G}_t(\mathbf{x},\mathbf{y})$ are $G$-invariant in $\mathbf{x}$ and $\mathbf{y}$. We list below further inequalities for the kernel $h_t$  proved in \cite[Section 4]{conjugate}. For every nonnegative integer $m$ and for any multi-indices $\alpha,\beta$, there are constants $C,c>0$ such
\begin{equation}\label{eq:DtDxDyHeat2t}
\bigl|\hspace{.25mm}\partial_t^m\partial_{\mathbf{x}}^{\alpha}\partial_{\mathbf{y}}^{\beta}h_t(\mathbf{x},\mathbf{y})\bigr|\le C\,t^{-m-\frac{|\alpha|}2-\frac{|\beta|}2}h_{2t}(\mathbf x,\mathbf y).
\end{equation}
Hence,
\begin{equation}\label{eq:DtDxDyHeat}
\bigl|\hspace{.25mm}\partial_t^m\partial_{\mathbf{x}}^{\alpha}\partial_{\mathbf{y}}^{\beta}h_t(\mathbf{x},\mathbf{y})\bigr|\le C\,t^{-m-\frac{|\alpha|}2-\frac{|\beta|}2}\mathcal G_{t\slash c}(\mathbf x,\mathbf y).
\end{equation}
 Moreover, if $\|\mathbf y-\mathbf y'\|\leq \sqrt{t}$, then
  \begin{equation}\label{eq:heat_Holder}
  \begin{split}|\partial_t^m  h_t(\mathbf{x},\mathbf{y}) & -
 \partial_t^m  h_t(\mathbf{x},\mathbf{y'}) | \leq C t^{-m} \frac{\|\mathbf y-\mathbf y'\|}{\sqrt{t}} \mathcal G_{t\slash c} (\mathbf x,\mathbf y).
  \end{split}\end{equation}

\section{Estimates of the Dunkl heat kernel}\label{sec:further}
The main goal of this section is to improve the estimates~\eqref{eq:DtDxDyHeat} (see Theorem \ref{teo:heat_new}). Then, using Theorem~\ref{teo:heat_new}, we conclude  bounds  for the Poisson kernel and for the Dunkl translations of radial compactly supported continuous functions.

\begin{theorem}\label{teo:heat_new}   For every nonnegative integer $m$ and every multi-indices $\alpha,\beta$ there are constants $C_{m,\alpha,\beta}, c>0$ such that
  \begin{equation}\label{eq:heat2} |\partial_t^m \partial_{\mathbf x}^{\alpha}\partial_{\mathbf y}^{\beta} h_t(\mathbf{x},\mathbf{y})|
  \leq C_{m,\alpha,\beta} t^{-m-\frac{|\alpha|}{2}-\frac{|\beta|}{2}} \Big(1+\frac{\| \mathbf x-\mathbf y\|}{\sqrt{t}}\Big)^{-2} \mathcal G_{t\slash c} (\mathbf x,\mathbf y).
  \end{equation}
  Moreover, if $\|\mathbf y-\mathbf y'\|\leq \sqrt{t}$, then
  \begin{equation}\label{eq:heat3}
  \begin{split}|\partial_t^m  h_t(\mathbf{x},\mathbf{y}) & -
 \partial_t^m  h_t(\mathbf{x},\mathbf{y'})|  \leq C_{m} t^{-m} \frac{\|\mathbf y-\mathbf y'\|}{\sqrt{t}}\Big(1+\frac{\| \mathbf x-\mathbf y\|}{\sqrt{t}}\Big)^{-2} \mathcal G_{t\slash c} (\mathbf x,\mathbf y).
  \end{split}\end{equation}
\end{theorem}

\begin{remark}
  \normalfont Observe that the estimates \eqref{eq:DtDxDyHeat} and \eqref{eq:heat2} differ by the factor $(1+\| \mathbf x-\mathbf y\|\slash \sqrt{t})^{-2}$. We want to emphasize that the presence of the factor is crucial in the proof of  the atomic decomposition stated in Theorem~\ref{teo:main1}.
\end{remark}

\begin{lemma} For all $\mathbf x,\mathbf y\in\mathbb R^N$ and for any $t>0$ we have
 \begin{equation}\label{eq:T_j}
 T_{j,\mathbf x}h_t(\mathbf{x},\mathbf{y})=\frac{y_j-x_j}{2t}h_t(\mathbf{x},\mathbf{y}),
 \end{equation}
  \begin{align}\label{eq:T_jx2}
  T_{j,\mathbf x}^2 h_t(\mathbf{x},\mathbf{y})=\frac{(y_j-x_j)^2}{(2t)^2}h_t(\mathbf{x},\mathbf{y})-\frac{1}{2t}h_t(\mathbf{x},\mathbf{y})-\frac{1}{2t}\sum_{\alpha \in R}\frac{k(\alpha)}{2}\alpha_j^2h_t(\sigma_\alpha(\mathbf{x}),\mathbf{y}).
 \end{align}
\end{lemma}
\begin{proof}
 The function $\mathbf{x} \mapsto \exp\left(-\frac{\|\mathbf{x}\|^2+\|\mathbf{y}\|^2}{4t}\right)$ is $G$-invariant so,  by \eqref{eq:Leibniz},
 \begin{align*}
  c_k (2t)^{\mathbf{N}/2}T_{j,\mathbf x}h_t(\mathbf{x},\mathbf{y})=&\partial_{x_j}
  \left(\exp\left(-\frac{\|\mathbf{x}\|^{2}+\|\mathbf{y}\|^2}{4t}\right)\right)E\left(\mathbf{x},\frac{\mathbf{y}}{2t}\right)\\
  &+T_{j,\mathbf x} \left(E\left(\mathbf{x},\frac{\mathbf{y}}{2t}\right)\right)\exp\left(-\frac{\|\mathbf{x}\|^2+\|\mathbf{y}\|^2}{4t}\right)\\=&c_k (2t)^{\mathbf{N}/2}\left(\frac{-x_j}{2t}h_t(\mathbf{x},\mathbf{y})+\frac{y_j}{2t}h_t(\mathbf{x},\mathbf{y})\right),
 \end{align*}
where in the last equality we have used \eqref{eq:T_j_x}. Thus \eqref{eq:T_j} is established.

To prove \eqref{eq:T_jx2} we utilize  \eqref{eq:T_j} and get
 \begin{align*}
  T_{j,\mathbf x}^2 h_t(\mathbf{x},\mathbf{y})&=T_{j,\mathbf x} \left(\frac{y_j-x_j}{2t}h_t(\mathbf{x},\mathbf{y})\right)=\frac{y_j}{2t}T_{j,\mathbf x}h_t(\mathbf{x},\mathbf{y})-T_{j,\mathbf x}\left(\frac{x_j}{2t}h_t(\mathbf{x},\mathbf{y})\right)\\&=\frac{y_j}{2t}\frac{(y_j-x_j)}{2t}h_t(\mathbf{x},\mathbf{y})-S_j(\mathbf{x},\mathbf{y},t).
 \end{align*}
 Let  $(\sigma_\alpha(\mathbf{x}))_{j}$ denote the $j$-th coordinate of $\sigma_\alpha(\mathbf{x})$. Further,
 \begin{align*}
 &S_j(\mathbf{x},\mathbf{y},t)=\partial_{x_j}\left(\frac{x_j}{2t}h_t(\mathbf{x},\mathbf{y})\right)\\&+\sum_{\alpha \in R}\frac{k(\alpha)}{2}\alpha_j\frac{\frac{x_j}{2t}h_t(\mathbf{x},\mathbf{y})-\frac{(\sigma_\alpha(\mathbf{x}))_{j}}{2t}h_t(\sigma_{\alpha}(\mathbf{x}),\mathbf{y})}{\langle \alpha,\mathbf{x}\rangle}\\=&\frac{1}{2t}h_t(\mathbf{x},\mathbf{y})+\frac{x_j}{2t}\partial_{x_j}h_t(\mathbf{x},\mathbf{y})+\sum_{\alpha \in R}\frac{k(\alpha)}{2}\alpha_j\frac{\frac{x_j}{2t}h_t(\mathbf{x},\mathbf{y})-\frac{x_j}{2t}h_t(\sigma_\alpha(\mathbf{x}),\mathbf{y})}{\langle \alpha,\mathbf{x}\rangle}\\&+\sum_{\alpha \in R}\frac{k(\alpha)}{2}\alpha_j\frac{\frac{x_j}{2t}h_t(\sigma_\alpha(\mathbf{x}),\mathbf{y})-\frac{(\sigma_\alpha(\mathbf{x}))_{j}}{2t}h_t(\sigma_{\alpha}(\mathbf{x}),\mathbf{y})}{\langle \alpha,\mathbf{x}\rangle}
  \\=&\frac{1}{2t}h_t(\mathbf{x},\mathbf{y})+\frac{x_j}{2t}T_{j,\mathbf x}h_t(\mathbf{x},\mathbf{y})\\&+
  \sum_{\alpha \in R}\frac{k(\alpha)}{2}\alpha_j\frac{\frac{x_j}{2t}h_t(\sigma_{\alpha}(\mathbf{x}),\mathbf{y})-\frac{(\sigma_\alpha(\mathbf{x}))_{j}}{2t}h_t(\sigma_{\alpha}(\mathbf{x}),\mathbf{y})}{\langle \alpha,\mathbf{x}\rangle}.
\end{align*} Note that $x_j-(\sigma_\alpha(\mathbf{x}))_{j}=\langle \mathbf{x},\alpha \rangle \alpha_j$. Therefore
\begin{align*}
&S_j(\mathbf{x},\mathbf{y},t)=\frac{1}{2t}h_t(\mathbf{x},\mathbf{y})+\frac{x_j}{2t}\frac{(y_j-x_j)}{2t}h_t(\mathbf{x},\mathbf{y})+\frac{1}{2t}\sum_{\alpha \in R}\frac{k(\alpha)}{2}\alpha_j^2h_t(\sigma_{\alpha}(\mathbf{x}),\mathbf{y}).
\end{align*}
Finally,
\begin{align*}
 T_{j,\mathbf x}^2h_t(\mathbf{x},\mathbf{y})=&\frac{y_j}{2t}\frac{(y_j-x_j)}{2t}h_t(\mathbf{x},\mathbf{y})-\frac{1}{2t}h_t(\mathbf{x},\mathbf{y})\\&-\frac{x_j}{2t}\frac{(y_j-x_j)}{2t}h_t(\mathbf{x},\mathbf{y})-\frac{1}{2t}\sum_{\alpha \in R}\frac{k(\alpha)}{2}\alpha_j^2h_t(\sigma_{\alpha}(\mathbf{x}),\mathbf{y}).
\end{align*}
\end{proof}

\begin{proof}[Proof of Theorem~\ref{teo:heat_new}] Clearly,  $\Delta_{\mathbf x}  h_t(\mathbf{x},\mathbf{y})=\partial_t h_t(\mathbf{x},\mathbf{y})$. Hence,  summing  up \eqref{eq:T_jx2} over $j=1,2,\ldots,N$, we obtain
 \begin{align}\label{eq:heat}
  \partial_t h_t(\mathbf{x},\mathbf{y})
  =\frac{\|\mathbf{x}-\mathbf{y}\|^{2}}{(2t)^2}h_t(\mathbf{x},\mathbf{y})
  -\frac{N}{2t}h_t(\mathbf{x},\mathbf{y})-\frac{1}{2t}\sum_{\alpha \in R}k(\alpha)h_t(\sigma_{\alpha}(\mathbf{x}),\mathbf{y}).
 \end{align}
 Applying \eqref{eq:heat} together with~\eqref{eq:DtDxDyHeat2t} we get
\begin{equation}\label{eq:heat_better2t}
\Big(1+\frac{\| \mathbf x-\mathbf y\|}{\sqrt{t}} \Big)^2h_t(\mathbf x,\mathbf y)\lesssim  h_{2t}(\mathbf x,\mathbf y)+\sum_{\alpha \in R}k(\alpha)h_t(\sigma_{\alpha}(\mathbf{x}),\mathbf{y}).
\end{equation}
Using  \eqref{eq:heat_better2t} and \eqref{heat_radial} we obtain
 \begin{equation}\label{eq:heat_better}
 h_t(\mathbf x,\mathbf y)\lesssim \Big(1+\frac{\| \mathbf x-\mathbf y\|}{\sqrt{t}} \Big)^{-2} \mathcal G_{t\slash c}(\mathbf x,\mathbf y),
 \end{equation}
which
 completes the proof of~\eqref{eq:heat2} for $m=0$, $\alpha=\beta=0$. Now \eqref{eq:heat2} in its general form   is a direct consequence of \eqref{eq:DtDxDyHeat2t} and \eqref{eq:heat_better}.

The inequality \eqref{eq:heat3} can be proved in a similar way. To this end we repeat arguments from the proof of~\eqref{eq:heat_Holder} presented in ~\cite[Theorem 4.3 (b)]{conjugate} and  apply (at the very end) \eqref{eq:heat_better} to get the additional factor $(1+\|\mathbf x-\mathbf y\|\slash \sqrt{t})^{-2}$. We omit the details.
\end{proof}

 \begin{remark} \normalfont Let us note that iteration of the procedure presented in the proof of Theorem~\ref{teo:heat_new} may lead to an improvement of \eqref{eq:heat2}. Indeed, if we use \eqref{eq:heat_better2t} twice, then
\begin{align*}
&h_t(\mathbf x,\mathbf y)\lesssim
{ \Big(1+\frac{\| \mathbf x-\mathbf y\|}{\sqrt{t}}\Big)^{-4}
\Big(h_{4t}(\mathbf x,\mathbf y)
+\sum_{\alpha\in R} h_{2t}(\sigma_\alpha(\mathbf x),\mathbf y)\Big) }\\
&+{\Big(1+\frac{\| \mathbf x-\mathbf y\|}{\sqrt{t}}\Big)^{-2}}
{ \sum_{\alpha\in R} k(\alpha)\Big(1+\frac{\| \sigma_\alpha(\mathbf x)- \mathbf y\|}{\sqrt{t}}\Big)^{-2} h_{2t}(\sigma_\alpha(\mathbf x),\mathbf y)}\\
&+{\Big(1+\frac{\| \mathbf x-\mathbf y\|}{\sqrt{t}}\Big)^{-2}}
{ \sum_{\alpha\in R} k(\alpha)\Big(1+\frac{\| \sigma_\alpha(\mathbf x)- \mathbf y\|}{\sqrt{t}}\Big)^{-2} \sum_{\beta\in R} h_t(\sigma_\beta(\sigma_\alpha(\mathbf x),\mathbf y)}\\
\lesssim & \Big(1+\frac{\| \mathbf x-\mathbf y\|}{\sqrt{t}}\Big)^{-4}\mathcal G_{t\slash c}(\mathbf x,\mathbf y)\\&+\Big\{\Big(1+\frac{\| \mathbf x-\mathbf y\|}{\sqrt{t}}\Big)^{-2} \sum_{\alpha\in R} k(\alpha)\Big(1+\frac{\| \sigma_\alpha(\mathbf x)- \mathbf y\|}{\sqrt{t}}\Big)^{-2}\Big\}\mathcal G_{t\slash c}(\mathbf x,\mathbf y).
\end{align*}
In particular for the root system  $A_2$, if $\mathbf x_1,..., \mathbf x_5$, and   $\mathbf y$ are located as in Figure~\ref{Fig1}, then
\begin{equation*}
\begin{split}
h_t(\mathbf x_j,\mathbf y)& \lesssim w(B(\mathbf y,\sqrt{t}))^{-1}\Big(1+\frac{\| \mathbf x_j-\mathbf y\|}{\sqrt{t}}\Big)^{-2} \ \ \text{\rm for} \ j=1,2,3;\\
h_t(\mathbf x_j,\mathbf y)& \lesssim w(B(\mathbf y,\sqrt{t}))^{-1}\Big(1+\frac{\| \mathbf x_j-\mathbf y\|}{\sqrt{t}}\Big)^{-4} \ \ \text{\rm for} \ j=4,5.\\
\end{split}
\end{equation*}

\begin{figure}[H]
\centering
\begin{tikzpicture}[scale=2.5]
\def\radius{0.7cm}
\def\offset{25}
\def\radiuss{0.015cm}

\foreach \angle in {0,...,2}
{
\draw[-{Latex[lightgray, length=7mm, width=5mm]},line width=0.3mm, lightgray] (0,0)--({cos(360/6*\angle)},{sin(360/6*\angle)});
};

\foreach \angle in {3,4,5}
{
\draw[-{Latex[lightgray, length=7mm, width=5mm,fill=white]},dotted,line width=0.3mm, lightgray] (0,0)--({cos(360/6*\angle)},{sin(360/6*\angle)});
};

\foreach \angle in {0,...,2}
{
\draw[fill=mygray, line width=0.3mm] ({\radius*cos(360/3*\angle+30+\offset)},{\radius*sin(360/3*\angle+30+\offset)}) circle (\radiuss);
\draw[fill=mygray, line width=0.3mm] ({\radius*cos(360/3*\angle+30-\offset)},{\radius*sin(360/3*\angle+30-\offset)}) circle (\radiuss);
};

\draw[fill=mygray, line width=0.3mm] ({1.15*\radius*cos(30+\offset)},{1.15*\radius*sin(30+\offset)}) node {$\mathbf{y}$};

\draw[fill=mygray, line width=0.3mm] ({1.15*\radius*cos(360/6*1+30+\offset+8)},{1.15*\radius*sin(360/6*1+30+\offset+8)}) node {$\mathbf{x}_1$};

\draw[fill=mygray, line width=0.3mm] ({1.15*\radius*cos(-360/6*1+30+\offset+10)},{1.15*\radius*sin(-360/6*1+30+\offset+10)}) node {$\mathbf{x}_2$};

\draw[fill=mygray, line width=0.3mm] ({1.15*\radius*cos(360/6*2+30+\offset)},{1.15*\radius*sin(360/6*2+30+\offset)}) node {$\mathbf{x}_5$};

\draw[fill=mygray, line width=0.3mm] ({1.15*\radius*cos(-360/6*2+30+\offset)},{1.15*\radius*sin(-360/6*2+30+\offset)}) node {$\mathbf{x}_4$};

\draw[fill=mygray, line width=0.3mm] ({1.15*\radius*cos(360/6*3+30+\offset+10)},{1.15*\radius*sin(360/6*3+30+\offset+10)}) node {$\mathbf{x}_3$};

\foreach \angle in {1,...,6}
{
\draw[line width=0.4mm]
(0,0)--({cos(360/6*\angle-30)},{sin(360/6*\angle-30)});
};
\end{tikzpicture}

\caption{The points $\mathbf{y}, \mathbf{x}_1,\ldots,\mathbf{x}_5$ and the root system $A_2$.}\label{Fig1}
\end{figure}
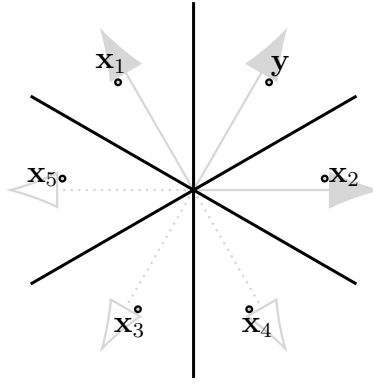

Let us also elaborate the product case  $A_1^N$, where $R=\{\pm\sqrt{2}e_j: j=1,...,N\}$ and  $e_j$ is the canonical  orthonormal basis in $\mathbb R^N$. If $\mathbf y=(1,1,...,1)$ and $\mathbf x=(\varepsilon_1,\varepsilon_2,...,\varepsilon_N)$, $\varepsilon_j\in\{-1,1\}$, then iteration of \eqref{eq:heat_better2t} leads to
\begin{equation}\label{product_h_t}
h_t(\mathbf x,\mathbf y)\lesssim w(B(\mathbf y,\sqrt{t}))^{-1} \Big(1+\frac{\| \mathbf x-\mathbf y\|}{\sqrt{t}}\Big)^{-4\ell},
\end{equation}
where $\ell=\#\{j: \varepsilon_j=-1\}$, which is exactly the smallest number of reflections $\sigma_{\sqrt{2}e_j}$ that are needed to pass from $\mathbf y$ to $\mathbf x$. If $k>0$, then~\eqref{product_h_t} is sharp, since the heat kernel $h_t(\mathbf x,\mathbf y)$ is the product of one dimensional heat kernels, whose behavior, in this case, is well understood (see e.g. \cite[Proposition 2.3]{ABDH}).
\end{remark}

\begin{corollary}\label{coro:corollary}Assume that $\Phi(\mathbf x)$ is a radial continuous function which is supported in $B(0,1)$. Let $\Phi_t(\mathbf x)=t^{-\mathbf N}\Phi(\mathbf x\slash t)$. There is a constant $C=C(\Phi)>0$ such that
   \begin{equation}\label{eq:compact}|\Phi_t(\mathbf x,\mathbf y)|\leq  CV(\mathbf x,\mathbf y, t)^{-1} \Big(1+\frac{\| \mathbf x-\mathbf y\|}{t}\Big)^{-2}\chi_{[0,1]}(d(\mathbf x,\mathbf y)\slash t).
   \end{equation}
\end{corollary}
\begin{proof}
There is a constant $C>0$ such that
$$|\Phi_t(\mathbf x)|\leq C h_{t^2}(\mathbf{x}),$$
where $h_t(\mathbf x)=c_k^{-1}\,(2t)^{-\mathbf{N}/2}\,e^{-\frac{{\|}\mathbf{x}{\|}^2}{4t}}$.
Applying \eqref{eq:translation-radial} we obtain
$$ |\Phi_t(\mathbf x,\mathbf y)|\leq Ch_{t^2}(\mathbf x,\mathbf y)\leq C'V(\mathbf x,\mathbf y, t)^{-1}\Big(1+\frac{\| \mathbf x-\mathbf y\|}{t}\Big)^{-2}.$$
By~\eqref{eq:translation-radial} and~\eqref{eq:A(x,y,eta)>d(x,y)} we have $\Phi_t(\mathbf x,\mathbf y)=0$ if $d(\mathbf x,\mathbf  y)>t$, so the  proof of \eqref{eq:compact} is complete.
\end{proof}

\noindent
{\bf Estimates for the Poisson kernel.} Let $p_t(\mathbf x,\mathbf y)$ denote the integral kernel of the operator $P_t=e^{-t\sqrt{-\Delta}}$.
It is related with the heat semigroup by the subordination formula
\begin{equation}\label{eq:subordination}
p_t(\mathbf x,\mathbf y)= \pi^{-1/2}\int_0^{\infty} e^{-u}h_{t^2\slash (4u)}(\mathbf x,\mathbf y) \frac{du}{\sqrt{u}}.
\end{equation}
The kernel $p_t(\mathbf x,\mathbf y)$ was introduced and
studied in~\cite{Roesler-Voit2}. It was called the $k$-{\it Cauchy kernel} there. For a continuous bounded function $f$ defined
on $\mathbb R^N$, the function $v(t,\mathbf x)=P_tf(\mathbf x)$, $v(0,\mathbf x)=f(\mathbf x)$,  solves the Cauchy problem $(\partial_t^2 +\Delta)u=0$, $v$ is continuous and bounded on $[0,\infty)\times \mathbb R^N$ (see~\cite[Theorem~5.6]{Roesler-Voit2}). }

Proposition 5.4 of~\cite{conjugate} asserts that there is a constant $C>0$ such that
\begin{equation}\label{eq:Poisson_up}
{p}_t(\mathbf{x},\mathbf{y})\leq\frac{C}{V(\mathbf{x},\mathbf{y},t+d(\mathbf{x},\mathbf{y}))}\,\frac{t}{t+d(\mathbf{x},\mathbf{y})}
\end{equation}
for every \,$t>0$ and for every \,$\mathbf{x},\mathbf{y}\in\mathbb{R}^N$.
Moreover, for any nonnegative integer \,$m$ and for any multi-index \,$\beta$, there is a constant \,$C\hspace{-.5mm}\ge\hspace{-.5mm}0$ such that, for every \,$t>0$ and for every \,$\mathbf{x},\mathbf{y}\in\mathbb{R}^N$,
\begin{equation}\label{DtDyPoisson}
\bigl|\hspace{.25mm}\partial_t^m\partial_{\mathbf{y}}^{\beta}\hspace{.25mm}p_t(\mathbf{x},\mathbf{y})\bigr|\le C\,p_t(\mathbf{x},\mathbf{y})\hspace{.25mm}\bigl(\hspace{.25mm}t\hspace{-.25mm}+d(\mathbf{x},\mathbf{y})\bigr)^{\hspace{-.5mm}-m-|\beta|}\times\begin{cases}
\,1&\text{if \,}m\hspace{-.25mm}=\hspace{-.25mm}0\hspace{.25mm},\\
\,1+\frac{d(\mathbf{x},\mathbf{y})}t&\text{if \,}m\hspace{-.5mm}>\hspace{-.5mm}0\hspace{.25mm}.\\
\end{cases}\end{equation}
The following proposition improves \eqref{eq:Poisson_up}.
\begin{proposition}\label{propo:Poisson}
If $N\geq 2$, then
\begin{equation}
\label{eq:Poisson_new}
 p_t(\mathbf x,\mathbf y)\lesssim \frac{t}{V(\mathbf x,\mathbf y, d(\mathbf x,\mathbf y)+t)}\cdot \frac{d(\mathbf x,\mathbf y)+t}{\| \mathbf x-\mathbf y\|^2+t^2}.
\end{equation}
If $N=1$, then
\begin{equation}
\label{eq:Poisson_dim1_new}
 p_t(\mathbf x,\mathbf y)\lesssim \frac{t}{V(\mathbf x,\mathbf y, d(\mathbf x,\mathbf y)+t)}\cdot \frac{d(\mathbf x,\mathbf y)+t}{\| \mathbf x-\mathbf y\|^2+t^2}\cdot \ln\Big( 1+\frac{\| \mathbf x-\mathbf y\|+t}{d(\mathbf x,\mathbf y)+t}   \Big).
\end{equation}
  \end{proposition}

\begin{proof}
The proof is similar to the proof of Proposition 6 of \cite{DP} and uses \eqref{eq:subordination} together with \eqref{eq:heat2}. We present the details. In order to prove~\eqref{eq:Poisson_new}, we first consider  the case  $d(\mathbf x,\mathbf y) \leq t$. In this case $d(\mathbf x,\mathbf y)+t\simeq t$.  If $\|\mathbf{x}-\mathbf{y}\|<t$ then~\eqref{eq:Poisson_new} reduces to \eqref{eq:Poisson_up}. If $\|\mathbf{x}-\mathbf{y}\| \geq t$ then by~\eqref{eq:subordination} and~\eqref{eq:heat2},
\begin{equation*}\begin{split}
p_t(\mathbf x,\mathbf y) \lesssim & w(B(\mathbf{x},t))^{-1}\int_{0}^{\infty} e^{-u}\frac{t^2/(4u)}{\|\mathbf{x}-\mathbf{y}\|^2+t^2/(4u)}\frac{w(B(\mathbf{x},t))}{w\Big(B\left( \mathbf x,\frac{t}{2\sqrt{u}}\Big)\right)} \frac{du}{\sqrt{u}} \\ \lesssim& w(B(\mathbf{x},t))^{-1}\int_{0}^{1/4} \frac{t^2/(4u)}{\|\mathbf{x}-\mathbf{y}\|^2} (\sqrt{u})^{N} \frac{du}{\sqrt{u}}\\&+w(B(\mathbf{x},t))^{-1}\int_{1/4}^{\infty}e^{-u} \frac{t^2}{\|\mathbf{x}-\mathbf{y}\|^2} (\sqrt{u})^{\mathbf{N}} \frac{du}{\sqrt{u}}
\\\lesssim & w(B(\mathbf{x},t))^{-1}\int_{0}^{1/4} \frac{t^2/(4u)}{\|\mathbf{x}-\mathbf{y}\|^2} (\sqrt{u})^{2} \frac{du}{\sqrt{u}}+w(B(\mathbf{x},t))^{-1} \frac{t^2}{\|\mathbf{x}-\mathbf{y}\|^2}
\\\lesssim & w(B(\mathbf{x},t))^{-1}\frac{t^2}{\|\mathbf{x}-\mathbf{y}\|^2}.
\end{split}\end{equation*}
Now we turn to the case $\|\mathbf{x}-\mathbf{y}\| \geq d(\mathbf x,\mathbf y)\geq t$. Then $d(\mathbf x,\mathbf y)+t\simeq d(\mathbf x,\mathbf y)$. Using Theorem \ref{teo:heat_new}, we have
    \begin{equation*}\begin{split}
    p_t(\mathbf x, \mathbf y)
    &\lesssim \int_0^{\infty} \frac{e^{-u}
    \exp (-4c ud(\mathbf x,\mathbf y)^2\slash t^2)}{w (B(\mathbf x,\frac{t}{2\sqrt{u}}))} \frac{t^2/(4u)}{\|\mathbf{x}-\mathbf{y}\|^2+t^2/(4u)} \frac{du}{\sqrt{u}} \\&= \int_0^{t^2\slash d(\mathbf x,\mathbf y)^2} +
    \int_{t^2\slash d(\mathbf x,\mathbf y)^2}^\infty = J_1 + J_2.
    \end{split}\end{equation*}
Further, since $N\geq 2$,
\begin{equation*}\label{eq:int_1}\begin{split}
J_1  &\lesssim  w(B(\mathbf{x},d(\mathbf{x},\mathbf{y})))^{-1}\int_0^{t^2\slash d(\mathbf x,\mathbf y)^2} \frac{w(B(\mathbf{x},d(\mathbf{x},\mathbf{y})))}{w (B(\mathbf x,\frac{t}{2\sqrt{u}}))} \frac{t^2/(4u)}{\|\mathbf{x}-\mathbf{y}\|^2+t^2/(4u)} \frac{du}{\sqrt{u}}\\
&\lesssim w(B(\mathbf{x},d(\mathbf{x},\mathbf{y})))^{-1}\int_0^{t^2\slash d(\mathbf x,\mathbf y)^2} \left(\frac{\sqrt{u}d(\mathbf{x},\mathbf{y})}{t}\right)^{N} \frac{t^2/(4u)}{\|\mathbf{x}-\mathbf{y}\|^2+t^2/(4u)} \frac{du}{\sqrt{u}}\\
&\lesssim w(B(\mathbf{x},d(\mathbf{x},\mathbf{y})))^{-1}\int_0^{t^2\slash d(\mathbf x,\mathbf y)^2} \left(\frac{\sqrt{u}d(\mathbf{x},\mathbf{y})}{t}\right)^{2} \frac{t^2/(4u)}{\|\mathbf{x}-\mathbf{y}\|^2} \frac{du}{\sqrt{u}}\\&\lesssim w(B(\mathbf{x},d(\mathbf{x},\mathbf{y})))^{-1}\int_0^{t^2\slash d(\mathbf x,\mathbf y)^2} \frac{d(\mathbf{x},\mathbf{y})^2}{\|\mathbf{x}-\mathbf{y}\|^2}\frac{du}{\sqrt{u}}\\&\lesssim w(B(\mathbf{x},d(\mathbf{x},\mathbf{y})))^{-1}\frac{td(\mathbf{x},\mathbf{y})}{\|\mathbf{x}-\mathbf{y}\|^2}.
\end{split}\end{equation*}
For $J_2$ we obtain
\begin{equation*}\label{eq:int_2}\begin{split}
J_2&\lesssim \frac{1}{w(B(\mathbf x,d(\mathbf x,\mathbf y)))}
\int\limits_{\frac{t^2}{d(\mathbf x,\mathbf y)^2}}^{\infty} e^{\frac{-4c ud(\mathbf x,\mathbf y)^2} {t^2}}\Big(\frac{2d(\mathbf x,\mathbf y)\sqrt{u}}{t}\Big)^{\mathbf N}\frac{t^2\, du}{4\|\mathbf{x}-\mathbf{y}\|^2u^{3/2}} \\&\lesssim
\frac{1}{w(B(\mathbf x,d(\mathbf x,\mathbf y)))}
\int\limits_{\frac{t^2}{d(\mathbf x,\mathbf y)^2}}^{\infty}\frac{t^2}{\|\mathbf{x}-\mathbf{y}\|^2} \frac{du}{u^{3/2}} \lesssim w(B(\mathbf x,d(\mathbf x,\mathbf y)))^{-1}\frac{td(\mathbf{x},\mathbf{y})}{\|\mathbf{x}-\mathbf{y}\|^2}.
\end{split}\end{equation*}
The proof of \eqref{eq:Poisson_dim1_new} goes in a similar way. We omit the details. \end{proof}

\section{Atomic decompositions - proof of Theorem \ref{teo:main1}}\label{sec:SectionDecomp}

\noindent
{\bf Inclusion $H^1_{(1,q)} \subseteq H^1_{\Delta}$.}  Let us remark that the inclusion $H^1_{(1,q)} \subseteq H^1_{\Delta}$ and the second inequality in \eqref{eq:main}
are easy consequences of~\eqref{eq:DtDxDyHeat} and~\eqref{eq:heat_Holder}.
  The proof is standard (see e.g.~\cite{St2}). To this end, it is enough to prove that there is $C>0$ such that $\|\mathcal{M}_{H}a\|_{L^1(dw)} \leq C$ for any $(1,q)$-atom. Let $a$ be an $(1,q)$-atom associated with $B(\mathbf{x}_0,r)$. It follows from~\eqref{heat_radial} that $\mathcal{M}_{H}$ is bounded on $L^q(dw)$, hence by H\"older's inequality and conditions~\eqref{numitem:support} and~\eqref{numitem:size} of Definition~\ref{def:C-W-atoms} we have
\begin{align*}
\|\mathcal{M}_Ha\|_{L^1(\mathcal{O}(B(\mathbf{x}_0,2r)),dw)} &\leq \|\mathcal{M}_Ha\|_{L^q(dw)}w(\mathcal{O}(B(\mathbf{x}_0,2r)))^{1-\frac{1}{q}} \\&\leq C\|a\|_{L^q(dw)}w(\mathcal{O}(B(\mathbf{x}_0,2r)))^{1-\frac{1}{q}} \leq C'.
\end{align*}
We now turn to estimate $\mathcal{M}_Ha$ on $\mathcal{O}(B(\mathbf{x}_0,2r))^c$. Using condition~\eqref{numitem:cancel} of Definition~\ref{def:C-W-atoms} and~\eqref{eq:heat_Holder} we get
\begin{align*}
|\mathcal{M}_Ha(\mathbf{x})| &\leq \sup_{\|\mathbf{x}-\mathbf{y}\|<t} \left|\int_{B(\mathbf{x}_0,r)} (h_{t^2}(\mathbf{y},\mathbf{z})-h_{t^2}(\mathbf{y},\mathbf{x}_0))a(\mathbf{z})\,dw(\mathbf{z})\right|\\& \leq C\sup_{d(\mathbf{x},\mathbf{y})<t}\int \left|\frac{r}{t}\mathcal{G}_{t^2/c}(\mathbf{y},\mathbf{x_0})a(\mathbf{z})\right|\,dw(\mathbf{z})  \\& \leq C\sup_{d(\mathbf{x},\mathbf{y})<t}\frac{r}{t}\mathcal{G}_{t^2/c}(\mathbf{y},\mathbf{x_0}),
\end{align*}
where in the last inequality we have used the fact that $\|a\|_{L^1(dw)} \leq 1$. It is not difficult to check that by~\eqref{eq:growth} and~\eqref{heat_radial} we have
\begin{align*}
\sup_{d(\mathbf{x},\mathbf{y})<t}\frac{r}{t}\mathcal{G}_{t^2/c}(\mathbf{y},\mathbf{x_0}) \leq C\frac{r}{d(\mathbf{x},\mathbf{x}_0)}w(B(\mathbf{x}_0,d(\mathbf{x},\mathbf{x}_0)))^{-1},
\end{align*}
which implies $\|\mathcal{M}_Ha\|_{L^1(\mathcal{O}(B(\mathbf{x}_0,2r))^{c},dw)} \leq C$. \qed

\noindent
{\bf Square function characterization of $H^1_{\Delta}$ and tent spaces.} Let
$$Q_tf=t\sqrt{-{\Delta}}e^{-t\sqrt{-{\Delta}}}f=t\frac{d}{dt} P_tf=\mathcal F^{-1} (t\|\xi\|e^{-t\|\xi\|} \mathcal Ff(\xi)).$$
The operators $Q_t$, initially defined on $L^1(dw)\cup L^2(dw)$,   have the form
$$ Q_tf(\mathbf x)=\int_{\mathbb R^N} q_t(\mathbf x,\mathbf y)f(\mathbf y)\, dw(\mathbf y),$$
where $q_t(\mathbf x,\mathbf y)=\frac{d}{dt}p_t(\mathbf x,\mathbf y)$. It can be easily deduced from \eqref{DtDyPoisson} that $|q_t(\mathbf x,\mathbf y)|\leq C p_t(\mathbf x,\mathbf y)$. Thus for every $1\leq p<\infty$, the operators $Q_t$ are uniformly bounded on $L^p({dw})$.

Consider the square function
\begin{equation}\label{eq:square}
Sf(\mathbf{x})=\left(\iint_{{\|}\mathbf{x}-\mathbf{y}{\|}<t}|Q_tf(\mathbf{y})|^2
\frac{dt\,{dw}(\mathbf{y})}{t\,{w}(B(\mathbf{x},t))}\right)^{1\slash 2}
\end{equation}
{and the space}
$
H^1_{\rm square}=\{ f\in L^1({dw})\,|\,\| Sf\|_{L^1({dw})}<\infty\}.
$
The following theorem was proved in \cite{conjugate}.
\begin{theorem}\label{teo:main5}
The spaces $H^1_{\Delta}$ and $H^1_{\rm square}$ coincide
and the corresponding norms $\| f\|_{H^1_{{\rm max}, H}}$ and $\|Sf\|_{L^1(dw)}$ are equivalent.
\end{theorem}

In order to prove our main result about atomic decomposition we use relation of $H^1_{{\rm square}}$ with the tent space $T_2^1$. The tent spaces were introduced in  Coifman, Meyer, and Stein~\cite{CMS}.
Recall  that a function $F$ defined on $\mathbb R_+\times \mathbb R^N$ is in the tent space $T_2^p$ if $\| F\|_{T_2^p}:=\| \mathcal AF(\mathbf x)\|_{L^p(dw)}<\infty$, where
$$\mathcal AF(\mathbf{x}) :=\Big( \int_0^\infty\int_{\|\mathbf{y}-\mathbf{x}\|<t} |F(t,\mathbf{y})|^2\frac{{dw}(\mathbf{y})}{{w}(B(\mathbf{x},t))}\frac{dt}{t}\Big)^{1\slash 2}.$$
So, $f\in H^1_{\rm square}$ if and only if  $F(t,\mathbf x)=Q_tf(\mathbf x)$ belongs to the tent space $T_2^1$.

A measurable function $A(t,\mathbf x)$ is said to be a  $T_2^1$-{\it atom} if there is a Euclidean ball $B=B(\mathbf y_0,r)$, such that

\noindent $\bullet$ $\supp A\subset \widehat B=\{(t,\mathbf x)\in \mathbb R_+\times \mathbb R^N\,|\, B(\mathbf x,t)\subset B(\mathbf y_0, r)\}$;

\noindent $\bullet$ $\int_0^\infty \int_{\mathbb R^N} |A(t,\mathbf x)|^2\, dw(\mathbf x)\frac{dt}{t}\leq w(B)^{-1}.$

It is well known that $F(t,\mathbf x)$ belongs to $T_2^1$ if and only if there are sequences $A_j$ of $T_2^1$-atoms and $\lambda_j\in\mathbb C$ such that
$$  \sum_{j=1}^{\infty} \lambda_j A_j=F,\ \ \ \sum_{j=1}^{\infty} |\lambda_j|\sim \| F\|_{T_2^1},$$
where the convergence is in $T_2^1$ norm and a.e. (see \cite{CMS} and \cite{Russ}).

\begin{remark}\label{rem:Russ}\normalfont
The functions $\lambda_jA_j(t,\mathbf x)$ can be taken of the form
$$\lambda_jA_j(t,\mathbf x)=F(t,\mathbf x)\chi_{S_j}(t,\mathbf x),$$
where $S_j\subset \mathbb R_+\times \mathbb R^N$ are mutually disjoint (see \cite{Russ}).
 \end{remark}
\noindent
{\bf Calder\'on reproducing formula.}
From now on we choose a radial function $\Psi\in C_c^\infty (B(0,1\slash 4))$ satisfying  $\int \Psi (\mathbf x)\, dw(\mathbf x)=0$ such that  the following Calder\'on  reproducing formula
\begin{align*}
f(\mathbf x)=\int_0^\infty \Psi_t(\mathbf x,\mathbf y)Q_tf(\mathbf y)\, dw(\mathbf y)\frac{dt}{t}
\end{align*}
holds for all $f \in L^2(dw)$ with the convergence in $L^2(dw)$ (see~\cite{conjugate}). Let us recall that $\Psi_t(\mathbf x,\mathbf y)=\tau_{\mathbf x} \Psi_t(-\mathbf y)$, $\Psi_t(\mathbf x)=t^{-\mathbf N} \Psi (\mathbf x\slash t)$, $\int\Psi_t(\mathbf x,\mathbf y)\, dw(\mathbf y)=0$, and $\Psi_t(\mathbf x,\mathbf y)=0$ if $d(\mathbf x,\mathbf y)>t$.

\noindent
{\bf Atomic decomposition  of $H^1_{\rm square}$ into $(1,2)$-atoms.}
We are now in a position to prove decomposition of $f\in H^1_{\rm square}=H^1_{\Delta}$ into $(1,2)$-atoms. We start our proof by assuming additionally that $f\in L^2(dw)$. Then this assumption is easily removed by the approximation argument presented in~\cite[Theorem 11.25, Lemma 13.8]{conjugate}. Set
\begin{align*}
\pi_{\Psi}F(\mathbf x)=\int_0^\infty \int_{\mathbb R^N} \Psi_t(\mathbf x,\mathbf y)F(t,\mathbf y)\, dw(\mathbf y)\frac{dt}{t}.
\end{align*}
Then $\| \pi_{\Psi} F\|_{L^2(dw)} \leq C \| F\|_{T_2^2}$.
Let $F(t,\mathbf x)=Q_tf(\mathbf x)$. Note that $F\in T_2^1\cap T_2^2$.  Applying atomic decomposition of $F$ as a function in $T_2^1$ combined with Remark \ref{rem:Russ}, we get
\begin{align*}
f(\mathbf x)=\pi_{\Psi} F(\mathbf x)=\sum_{j=1}^\infty \lambda_j \pi_\Psi A_j(\mathbf x),
\end{align*}
where the convergence is in $L^2(dw)$. Hence,  it suffices to show that there is a constant $C>0$ such that
 \begin{equation}\label{eq:A_to_a}
 \| \pi_{\Psi}A\|_{(H^1_{(1,2)}} \leq C
 \end{equation}
 for any $A(t,\mathbf x)$ being a $T_2^1$-atom.
To this end assume that $A$ is associated with $\widehat B$, where $B=B(\mathbf y_0,r)$. Set
\begin{align*}
& a(t,\mathbf x)=\int_{\mathbb R^N} \Psi_t(\mathbf x,\mathbf y)A(t,\mathbf y)\, dw(\mathbf y),\\
&g(\mathbf x)=\pi_\Psi A(\mathbf x)=
\int_0^\infty \int_{B(\mathbf y_0, r)} \Psi_t(\mathbf x,\mathbf y)A(t,\mathbf y)\, dw(\mathbf y)\frac{dt}{t}=\int_0^r a(t,\mathbf x)\, \frac{dt}{t}.
\end{align*}
Then $\| g\|_{L^2(dw)}\leq C \| A\|_{T_2^1}\leq \frac{C'}{w(B)^{1\slash 2}}$, $ \supp g\subset \mathcal O(B)=\bigcup\limits_{\sigma\in G} B(\sigma(\mathbf y_0), r)$, and   $\int g(\mathbf x)\, dw(\mathbf x)=0$.

We denote by $M_{{\rm HL}}$ the Hardy-Littlewood maximal function
\begin{align*}
M_{{\rm HL}} f(\mathbf x)=\sup_{B(\mathbf y,R)\ni \mathbf x} \frac{1}{w(B(\mathbf y,R))}\int_{B(\mathbf y,R)}|f(\mathbf y')|\, dw(\mathbf y').
\end{align*}

\begin{lemma}
\label{lem:pointwise_T12_atom}
Assume that $\sigma \in G$ is such that  $\|\sigma(\mathbf{y}_0)-\mathbf{y}_0\|>4r$. Then for $\mathbf{x}$ such that $\mathbf{x} \in B(\sigma(\mathbf{y}_0),r)$
 we have
\begin{align*}
|a(t,\mathbf{x})|\leq C\frac{t^2}{\|\sigma(\mathbf{y}_0)-\mathbf{y}_0\|^{2}}\sum_{\sigma' \in G}M_{{\rm HL}}(A(t,\cdot))(\sigma'(\mathbf{x})).
\end{align*}
The constant $C>0$ is independent of $A(t,\mathbf{x})$,  $\sigma \in G$, $\mathbf x \in \mathbb{R}^N$ and $t>0$.
\end{lemma}

\begin{proof}
For $(t, \mathbf{y}) \in \supp A \subset \widehat{B}$ and $\mathbf{x} \in B(\sigma(\mathbf{y}_0),r)$ we have $\|\mathbf{x}-\mathbf{y}\| \sim \|\mathbf{y}_0-\sigma(\mathbf{y}_0)\|$. Therefore, by Corollary~\ref{coro:corollary},
\begin{align*}
\left| a(t,\mathbf x)\right| &\lesssim \int_{\mathbb{R}^N} \frac{t^2}{\| \mathbf x-\mathbf y\|^2V(\mathbf x,\mathbf y, t)}\chi_{[0,1]}\left(\frac{d(\mathbf x,\mathbf y)}{t}\right)|A(t,\mathbf{y})|\,dw(\mathbf{y})\\& \lesssim \sum_{\sigma' \in G} \int_{B(\sigma'(\mathbf{x}),t)}V(\sigma'(\mathbf x),\mathbf y, t)^{-1} \frac{t^2}{\|\sigma(\mathbf{y}_0)-\mathbf{y}_0\|^{2}}|A(t,\mathbf{y})|\,dw(\mathbf{y})\\&\lesssim \left(\sum_{\sigma' \in G}M_{{\rm HL}}(A(t,\cdot))(\sigma'(\mathbf{x}))\right)\frac{t^2}{\|\sigma(\mathbf{y}_0)-\mathbf{y}_0\|^{2}}.
\end{align*}
\end{proof}

\begin{lemma}
\label{lem:L^2}
If $\|\sigma(\mathbf{y}_0)-\mathbf{y}_0\|>4r$, then
\begin{align*}
\|g\|_{L^2(B(\sigma(\mathbf{y}_0),r),dw)} \leq \frac{C}{w(B(\mathbf{y}_0,r))^{1/2}}\frac{r^2}{\|\sigma(\mathbf{y}_0)-\mathbf{y}_0\|^{2}}.
\end{align*}
The constant $C>0$ is independent of $A(t,\mathbf x)$ and $\sigma\in G$.
\end{lemma}

\begin{proof}
By the Minkowski integral inequality, Lemma~\ref{lem:pointwise_T12_atom}, and the  Cauchy-Schwarz inequality, we have
\begin{align*}
&\left(\int_{B(\sigma(\mathbf{y}_0),r)}|g(\mathbf{x})|^2 \,dw(\mathbf{x})\right)^{1/2} \\&\leq \int_0^{r}\left(\int_{B(\sigma(\mathbf{y}_0),r)}|a(t,\mathbf{x})|^2 \,dw(\mathbf{x})\right)^{1/2}\,\frac{dt}{t} \\&\lesssim \int_0^{r}\left(\int\limits_{B(\sigma(\mathbf{y}_0),r)}\left(\sum_{\sigma' \in G}M_{{\rm HL}}(A(t,\cdot))(\sigma'(\mathbf{x}))\right)^2\frac{t^4}{\|\sigma(\mathbf{y}_0)-\mathbf{y}_0\|^{4}}\,dw(\mathbf{x})\right)^{1/2}\,\frac{dt}{t} \\
&\lesssim \int_0^{r} \frac{t^2}{\|\sigma(\mathbf{y}_0)-\mathbf{y}_0\|^{2}}\sum_{\sigma' \in G}\left(\int_{\mathbb{R}^N}\big(M_{{\rm HL}}(A(t,\cdot))(\sigma'(\mathbf{x}))\big)^2\,dw(\mathbf{x})\right)^{1/2}\,\frac{dt}{t} \\
&\lesssim \int_0^{r} \frac{t^2}{\|\sigma(\mathbf{y}_0)-\mathbf{y}_0\|^{2}}
\left(\int_{\mathbb{R}^N}|(A(t,\mathbf{x})|^2\,dw(\mathbf{x})\right)^{1/2}\,\frac{dt}{t} \\
&\lesssim \left(\int_0^{r}\frac{t^4}{\|\sigma(\mathbf{y}_0)-\mathbf{y}_0\|^{4}}\,\frac{dt}{t}\right)^{1/2}
\left(\int_0^{r}\int_{\mathbb{R}^N}|A(t,\mathbf{x})|^2\,dw(\mathbf{x})\,\frac{dt}{t}\right)^{1/2} \\&\lesssim \frac{r^2}{\|\sigma(\mathbf{y}_0)-\mathbf{y}_0\|^{2}}\frac{1}{w(B(\mathbf{y}_0,r))^{1/2}}.
\end{align*}
\end{proof}

\begin{proposition}
\label{propo:main}
There exists $C>0$ independent of $A(t,\mathbf{x})$ such that  $g=\pi_{\Psi}A$ can be written as
\begin{align*}
 g=\sum_{j=1}^{\infty}\lambda_j a_j,
\end{align*}
where $a_j$ are $(1,2)$-atoms and $\sum_{j=1}^{\infty}|\lambda_j| \leq C$.
\end{proposition}
\begin{proof}
Let $\sigma_0=e$ and $G=\{\sigma_0,\sigma_1,\sigma_2,\ldots,\sigma_{|G|-1}\}$. We know that $g=\pi_{\Psi} A$ is supported by
$$\mathcal O(B)=\bigcup_{j=0}^{|G|-1} B(\sigma_j(\mathbf y_0), r).$$
Set $E_0=B(\mathbf{y}_0,r)$,
\begin{align*}
 E_j=B(\sigma_j(\mathbf{y}_0),r) \setminus \bigcup_{i=0}^{j-1}B(\sigma_i(\mathbf{y}_0),r) \qquad \text{ for }j=1,2,\ldots, |G|-1,
\end{align*}
and  $g_j=g\chi_{E_j}$. Then $g=\sum_{j=0}^{|G|-1} g_j$, $\supp g_j\subset E_j\subseteq B(\sigma_j(\mathbf y_0), r))$.
Define
$$\mathcal I=\{j\in \{1,2,...,|G|-1\}\,|\,\|\sigma_j(\mathbf y_0)-\mathbf y_0\| \geq 4r\}, \ \ \mathcal J=\{0,1,...,|G|-1\}\setminus \mathcal I.$$
For  $j\in \mathcal I$ let
$m_j=\left\lfloor (\|\sigma_j(\mathbf{y}_0)-\mathbf{y}_0\|)\slash {r}\right\rfloor $.  Set
\begin{align*}
\mathbf{x}_n^{\{j\}}= \sigma_j(\mathbf y_0)+n\frac{\mathbf y_0-\sigma_j(\mathbf y_0)}{m_j} \qquad \text{ for } n=0,1,...,m_j.
\end{align*}
Then $r\leq \| \mathbf x_n^{\{j\}}-\mathbf x_{n+1}^{\{j\}}\|\leq 2r$. Let $c_j=\int_{\mathbb{R}^N} g_j(\mathbf{x})\,dw(\mathbf{x})$. By Lemma~\ref{lem:L^2} and the Cauchy--Schwarz inequality we have
\begin{equation}
\label{eq:Lemma_L^2_application}
|c_j| \leq C_1\frac{r^2}{\|\sigma_j(\mathbf{y}_0)-\mathbf{y}_0\|^2}.
\end{equation}
Set
\begin{align*}
a_0^{\{j\}}=\frac{\|\sigma_j(\mathbf{y}_0)-\mathbf{y}_0\|^2}{r^2}\left(g_j-c_j\frac{1}{w\big(B(\mathbf{x}_1^{\{j\}},r)\big)}
\chi_{B(\mathbf{x}_1^{\{j\}},r)}\right),
\end{align*}
\begin{align*}
a_n^{\{j\}}=&c_j\frac{\|\sigma_j(\mathbf{y}_0)-\mathbf{y}_0\|^2}{r^2}\frac{1}{w\big(B(\mathbf{x}_{n}^{\{j\}},r)\big)}
\chi_{B(\mathbf{x}_{n}^{\{j\}},r)}\\&-c_j\frac{\|\sigma_j(\mathbf{y}_0)-\mathbf{y}_0\|^2}{r^2}\frac{1}{w\big(B(\mathbf{x}_{n+1}^{\{j\}},r)\big)}\chi_{B(\mathbf{x}_{n+1}^{\{j\}},r)}
\end{align*}
for $n=1,2,\ldots,m_j-1$, and
\begin{align*}
b_j=c_j\frac{1}{w(B(\mathbf{x}_{m_j},r))}\chi_{B(\mathbf{x}_{m_j},r)}=c_j\frac{1}{w(B(\mathbf y_0,r))}\chi_{B(\mathbf y_0,r)}.
\end{align*}
The positions of {$B\big(\mathbf x_n^{\{j\}}, r\big)$} for the root system $A_2$ are schematized in Figure~\ref{Fig}.

\begin{figure}[H]
\centering
\begin{tikzpicture}[scale=3]
\def\radius{0.7cm}
\def\offset{10}
\def\radiuss{0.09cm}

\foreach \angle in {1,...,6}
{
\draw[line width=0.5mm, gray]
(0,0)--({cos(360/6*\angle-30)},{sin(360/6*\angle-30)});
};

\foreach \angle in {0,...,2}
{
\draw[-{Latex[length=7mm, width=5mm]},line width=0.5mm] (0,0)--({cos(360/6*\angle)},{sin(360/6*\angle)});
};

\foreach \angle in {3,4,5}
{
\draw[-{Latex[length=7mm, width=5mm,fill=white]},dotted,line width=0.5mm] (0,0)--({cos(360/6*\angle)},{sin(360/6*\angle)});
};

\foreach \angle in {0,...,2}
{
\draw[fill=mygray, line width=0.3mm] ({\radius*cos(360/3*\angle+30+\offset)},{\radius*sin(360/3*\angle+30+\offset)}) circle (\radiuss);
\draw[fill=mygray, line width=0.3mm] ({\radius*cos(360/3*\angle+30-\offset)},{\radius*sin(360/3*\angle+30-\offset)}) circle (\radiuss);
};

\draw[fill=mygray, line width=0.3mm] ({1.3*\radius*cos(30+\offset)},{1.3*\radius*sin(30+\offset)}) node {$B(\mathbf{y}_0,r)$};

\foreach \angle in {1,2}
{
\pgfmathsetmacro{\number}{floor((veclen(({\radius*cos(360/3*\angle+30+\offset)-\radius*cos(30+\offset)},{\radius*sin(360/3*\angle+30+\offset)-\radius*sin(30+\offset)})))/(1.4*\radiuss)+1)}

\pgfmathsetmacro{\step}{veclen(({\radius*cos(360/3*\angle+30+\offset)-\radius*cos(30+\offset)},{\radius*sin(360/3*\angle+30+\offset)-\radius*sin(30+\offset)}))/(\number*19.5*2)}

\foreach \x in {1,...,\number}
{
\draw ({\radius*cos(360/3*\angle+30+\offset)},{\radius*sin(360/3*\angle+30+\offset)})+({\step*\x*(\radius*cos(30+\offset)-\radius*cos(360/3*\angle+30+\offset))},{\step*\x*(\radius*sin(30+\offset)-\radius*sin(360/3*\angle+30+\offset))}) circle (\radiuss);
};

};

\foreach \angle in {1}
{
\pgfmathsetmacro{\number}{floor((veclen(({\radius*cos(360/3*\angle+30-\offset)-\radius*cos(30-\offset)},{\radius*sin(360/3*\angle+30-\offset)-\radius*sin(30+\offset)})))/(1.4*\radiuss))}

\pgfmathsetmacro{\step}{veclen(({\radius*cos(360/3*\angle+30-\offset)-\radius*cos(30+\offset)},{\radius*sin(360/3*\angle+30-\offset)-\radius*sin(30+\offset)}))/(\number*17*2)}

\foreach \x in {1,...,\number}
{
\draw ({\radius*cos(360/3*\angle+30-\offset)},{\radius*sin(360/3*\angle+30-\offset)})+({\step*\x*(\radius*cos(30+\offset)-\radius*cos(360/3*\angle+30-\offset))},{\step*\x*(\radius*sin(30+\offset)-\radius*sin(360/3*\angle+30-\offset))}) circle (\radiuss);
};

\draw[fill=black, line width=0.3mm] (0,0) circle (0.4mm);

};

\foreach \angle in {2}
{
\pgfmathsetmacro{\number}{floor((veclen(({\radius*cos(360/3*\angle+30-\offset)-\radius*cos(30-\offset)},{\radius*sin(360/3*\angle+30-\offset)-\radius*sin(30+\offset)})))/(1.4*\radiuss))}

\pgfmathsetmacro{\step}{veclen(({\radius*cos(360/3*\angle+30-\offset)-\radius*cos(30+\offset)},{\radius*sin(360/3*\angle+30-\offset)-\radius*sin(30+\offset)}))/(\number*19.5*2)}

\foreach \x in {1,...,\number}
{
\draw ({\radius*cos(360/3*\angle+30-\offset)},{\radius*sin(360/3*\angle+30-\offset)})+({\step*\x*(\radius*cos(30+\offset)-\radius*cos(360/3*\angle+30-\offset))},{\step*\x*(\radius*sin(30+\offset)-\radius*sin(360/3*\angle+30-\offset))}) circle (\radiuss);
};

\draw[fill=black, line width=0.3mm] (0,0) circle (0.4mm);

};
\end{tikzpicture}
\caption{ The balls $B(\mathbf x_n^{\{j\}}, r)$ for the root system $A_2$.}
\label{Fig}
\end{figure}
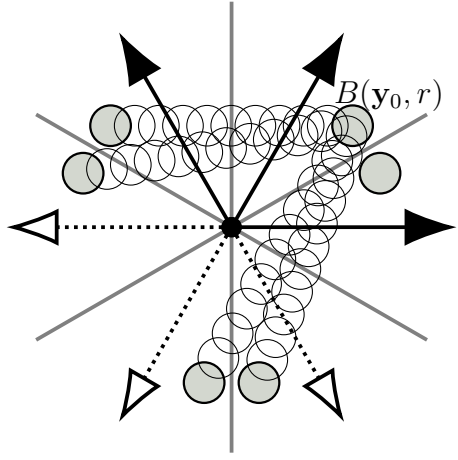
\noindent
Clearly,
\begin{align*}
g_j=\sum_{n=0}^{m_j-1}\frac{r^2}{\|\sigma_j(\mathbf{y}_0)-\mathbf{y}_0\|^2}a_n^{\{j\}}+b_j.
\end{align*}
It follows from Lemma~\ref{lem:L^2} and the doubling property that each function $a_n^{\{j\}}$ (for $j=0,1,\ldots, m_j-1$) is a multiple of a $(1,2)$-atom associated with the ball $B\big(\mathbf x_n^{\{j\}},4 r \big)$. The multiplicity constant depends on $\Psi$ and the doubling constant  but it is independent of $A(t,\mathbf x)$.
Write
\begin{align*}
a=\sum_{j\in \mathcal J} g_j+ \sum_{j\in \mathcal I} b_j.
\end{align*}
Then
\begin{align}\label{eq:decomposition}
g&=\sum_{j=0}^{|G|-1}g_j=a+\sum_{j\in \mathcal I}\sum_{i=0}^{m_j-1}\frac{r^2}{\|\sigma_j(\mathbf{y}_0)-\mathbf{y}_0\|^2}a_{i}^{\{j\}}.
\end{align}
Note that by the  construction above $\supp a \subset B(\mathbf{y}_0,16r)$ and, by Lemma~\ref{lem:L^2} and~\eqref{eq:Lemma_L^2_application}, we have
\begin{align*}
\|a\|_{L^2(dw)}\leq \sum_{j\in \mathcal J} \| g_j\|_{L^2(dw)}+\sum_{j\in \mathcal I} \| b_j\|_{L^2(dw)} \leq C_2\frac{1}{w(B(\mathbf{y}_0,16r))^{1/2}}.
\end{align*}
Moreover, $\int a(\mathbf x)\, dw(\mathbf x)=0$, since $\int g(\mathbf x)\, dw(\mathbf x)=0$ and $\int a_i^{\{j\}}(\mathbf x)\, dw(\mathbf x)=0$.
So the function $a$ is a~multiple of a $(1,2)$-atom. Further,
\begin{align*}
\sum_{j\in \mathcal I} \sum_{i=0}^{m_j-1}\frac{r^2}{\|\sigma_j(\mathbf{y}_0)-\mathbf{y}_0\|^2} \leq \sum_{j\in \mathcal I}\frac{r^2}{\|\sigma_j(\mathbf{y}_0)-\mathbf{y}_0\|^2}m_j \leq |G|\slash 4.
\end{align*}
Therefore~\eqref{eq:decomposition} is the desired atomic decomposition.
\end{proof}
Thus we have proved Theorem \ref{teo:main1}  for $q=2$.

\noindent
{\bf Decomposition into $(1,\infty)$ atoms.}  To finish the proof of Theorem \ref{teo:main1} it suffices to refer to the following known proposition. For the convenience of the reader we provide a short proof.
\begin{proposition}
There is a constant $C>0$ such that any $(1,2)$-atom $a(\mathbf{x})$ can be written as
\begin{align*}
a(\mathbf{x})=\sum_{j=1}^{\infty}\lambda_jb_j(\mathbf{x}),
\end{align*}
where $b_j$ are $(1,\infty)$-atoms and $\sum_{j=1}^{\infty}|\lambda_j| \leq C$.
\end{proposition}

\begin{proof}
Fix a $(1,2)$-atom $a(\mathbf{x})$. Since the measure $dw$ is doubling, without loosing of generality we can assume that $a(\mathbf x)$ is associated with a cube $Q$, i.e.
\begin{equation}
\label{eq:atom_with_cube}
\supp a \subset Q, \qquad \|a\|_{L^2(dw)}\leq w(Q)^{-1/2}, \qquad \int_{\mathbb{R}^N}a\,dw=0.
\end{equation}
Clearly, there is a constant $C_1>1$, which depends on the doubling constant and $N$, such that
$w(Q)\leq C_1 w(Q')$, where $Q'$ is any sub-cube of $Q$, $\ell (Q')=\ell (Q)\slash 2$, where $\ell (Q)$ denote the side length of $Q$.
Form the Calder\'on--Zygmund decomposition of  $|a|^2$ at height $\lambda=\varepsilon^{-2}w(Q)^{-2}$, where $\varepsilon=4^{-1}C_1^{-1/2}$.  This yields a sequence of disjoint  cubes $Q_j\subseteq Q$ such that
\begin{equation}
\label{eq:Calderon_Zygmund_properties}
w(Q_j) \leq \lambda^{-1}\int_{Q_j}|a(\mathbf{x})|^2\,dw(\mathbf{x}) <C_1 w(Q_j),
\end{equation}
\begin{equation}\label{eq:Calderon_Zygmund_properties_2}
|a(\mathbf{x})|^2 \leq \lambda \text{ for }\mathbf{x}\in \Omega=Q \setminus \bigcup_{j=1}^{\infty}Q_j.
\end{equation}
Set $a_{Q_j}=\frac{1}{w(Q_j)}\int_{Q_j}a(\mathbf{x})\,dw(\mathbf{x})$.
We write
\begin{align*}
a=\sum_{j=1}^{\infty}(a-a_{Q_j})\chi_{Q_j}+\Big(a\chi_{\Omega}+\sum_{j=1}^{\infty}a_{Q_j}\chi_{Q_j}\Big)=\sum_{j=1}^{\infty}a_j+b_1.
\end{align*}
We first prove that $b_1$ is a multiple of a $(1,\infty)$-atom associated with $Q$. Clearly, $\supp b_1 \subset Q$ and $\int_{Q}b_1=0$. Moreover,
\begin{align*}
|a_{Q_j}|=\left|\frac{1}{w(Q_j)}\int_{Q_j}a(\mathbf{x})\,dw(\mathbf{x})\right| \leq C_1^{1/2}\lambda^{1/2}=C_1^{1/2}\varepsilon^{-1}w(Q)^{-1}.
\end{align*}
Therefore, by~\eqref{eq:Calderon_Zygmund_properties_2} and~\eqref{eq:Calderon_Zygmund_properties},
\begin{align*}
|b_1(\mathbf{x})|\leq \varepsilon^{-1}w(Q)^{-1}+C_1^{1/2}\frac{1}{\varepsilon w(Q)}\leq (1+C_1^{1\slash 2})\varepsilon^{-1} w(Q)^{-1}=C_2 w(Q)^{-1}.
\end{align*}
Next, we show that $a_j$ is a multiple of a $(1,2)$-atom associated with $Q_j$. Obviously, $\supp a_j \subset Q_j$ and $\int_{Q_j}a_j=0$. Furthermore, by~\eqref{eq:Calderon_Zygmund_properties},
\begin{align*}
\|a_j\|_{L^2(dw)} \leq 2C_1^{1/2}\lambda^{1/2}w(Q_j)^{1/2}=2C_1^{1/2}\varepsilon^{-1}w(Q)^{-1}w(Q_j)^{1/2},
\end{align*}
which implies
\begin{align*}
\|a_j\|_{H^1_{(1,2)}}\leq 2C_1^{1/2}\varepsilon^{-1}w(Q)^{-1}w(Q_j).
\end{align*}
Finally, note that by~\eqref{eq:atom_with_cube} and~\eqref{eq:Calderon_Zygmund_properties},
\begin{align*}
\sum_{j=1}^{\infty}w(Q_j) \leq \lambda^{-1}\sum_{j=1}^{\infty}\int_{Q_j}|a(\mathbf{x})|^2\,dw(\mathbf{x})\leq \lambda^{-1}\|a\|^{2}_{L^2(dw)} \leq \frac{1}{\lambda w(Q)}=\varepsilon^{2}w(Q),
\end{align*}
which implies
\begin{align*}
\sum_{j=1}^{\infty}\|a_j\|_{H^1_{(1,2)}}\leq 2C_1^{1/2}\varepsilon^{-1}w(Q)^{-1}\sum_{j=1}^{\infty}w(Q_j) \leq 2C_1^{1/2}\varepsilon = 1\slash 2.
\end{align*}
 We repeat the argument above with $a_j$ in place of $a$. Iterating this procedure, we obtain a representation $a(\mathbf{x})=\sum_{j=1}^{\infty}\lambda_jb_j(\mathbf{x})$, where $b_j$ are $(1,\infty)$-atoms and $\sum_{j=1}^{\infty}|\lambda_j| \leq 2C_2$.
\end{proof}

\subsection*{Acknowledgements}
The authors want to thank Jean-Philippe Anker, Detlef M\"uller, Margit R\"osler, and Michael Voit for conversations on the subject of the paper. The authors are also greatly indebted to the
referees for  helpful comments which have improved the presentation of the paper.

\end{document}